\documentclass[a4paper,reqno, 11pt]{amsart}   	
\usepackage[DIV=12, oneside]{typearea}			
\usepackage[utf8]{inputenc}
\usepackage[T1]{fontenc}
\usepackage[english]{babel}
\usepackage[centertags]{amsmath}
\usepackage{amstext,amssymb,amsopn,amsthm}
\usepackage{mathrsfs}
\usepackage{dsfont}
\usepackage{bbm}
\usepackage{thmtools}
\usepackage{graphicx}
\usepackage[titletoc,title]{appendix}
\usepackage{etexcmds}

\usepackage{multirow}

\usepackage{tikz}
\usepackage{pgfplots}
\pgfplotsset{compat=newest}
\usetikzlibrary{calc, patterns,angles,quotes, tikzmark, fit, shapes.geometric}
\tikzset{%
	ebo unit/.store in=\ebounit,
	ebo corners/.style={rounded corners=#1\ebounit},
}
\usepgfplotslibrary{fillbetween}
\definecolor{c595959}{RGB}{89,89,89}
\definecolor{c5a5a5a}{RGB}{90,90,90}

\usepackage[colorlinks=true, linkcolor=black, urlcolor={black},
pdfborder={0 0 0}, citecolor=black]{hyperref}

\usepackage{enumitem}
\setlist[enumerate]{itemsep=0mm}

\addto\extrasenglish{}
\addto\extrasenglish{}
\addto\extrasenglish{}

\theoremstyle{plain}
\declaretheorem[title=Theorem, parent=section]{theorem}
\declaretheorem[title=Lemma,sibling=theorem]{lemma}
\declaretheorem[title=Proposition,sibling=theorem]{proposition}
\declaretheorem[title=Corollary,sibling=theorem]{corollary}

\theoremstyle{definition}

\declaretheorem[title=Remark,sibling=theorem]{remark}
\declaretheorem[title=Remark, numbered=no]{remark*}
\declaretheorem[title=Example, sibling=theorem]{example}
\declaretheorem[title=Counterexample, sibling=theorem]{counterexample}
\declaretheorem[title=Assumption, numbered=no]{assumption*}

\numberwithin{equation}{section}

\newcommand{\N}{\mathds{N}}
\newcommand{\R}{\mathds{R}}

\def\hmath$#1${\texorpdfstring{{\rmfamily\textit{#1}}}{#1}}

\newcommand{\cC}{\mathcal{C}}

\newcommand{\cE}{\mathcal{E}}

\newcommand{\eps}{\varepsilon}

\newcommand{\BIGOP}[1]
{
	\mathop{\mathchoice%
		{\raise-0.22em\hbox{\huge $#1$}}%
		{\raise-0.05em\hbox{\Large $#1$}}{\hbox{\large $#1$}}{#1}}}

\def\XXint#1#2#3{{\setbox0=\hbox{$#1{#2#3}{\int}$}
		\vcenter{\hbox{$#2#3$}}\kern-.5\wd0}}

\newcommand{\BIGboxplus}{\mathop{\mathchoice%
		{\raise-0.35em\hbox{\huge $\boxplus$}}%
		{\raise-0.15em\hbox{\Large $\boxplus$}}{\hbox{\large $\boxplus$}}{\boxplus}}}

\newcommand{\I}[1]{(\text{I}_{#1})}
\newcommand{\II} [1] {(\text{II}_{#1})}
\newcommand{\III}[1]{(\text{III}_{#1})}

\DeclareMathOperator{\diam}{diam}

\DeclareMathOperator{\esssup}{esssup}

\DeclareMathOperator{\tail}{Tail}
\DeclareMathOperator{\pv}{p.v.}

\renewcommand{\d}{\textnormal{d}}

\newcommand{\1}{{\mathbbm{1}}}

\usepackage{esint}
\newcommand{\norm}[1]{\left\lVert#1\right\rVert} 
\newcommand{\abs}[1]{\ensuremath{\left\vert#1\right\vert}} 
 
\makeatletter
\def\namedlabel#1#2{\begingroup
	#2%
	\def\@currentlabel{#2}%
	\phantomsection\label{#1}\endgroup
}
\makeatother

\newcommand{\cdini}{C^{1,\text{\normalfont{Dini}}}}

\begin{document}
	\allowdisplaybreaks
	\title{The inhomogeneous fractional Dirichlet problem} 
	
	\author{Florian Grube}

	\address{Fakult{\"a}t f{\"u}r Mathematik, Universit{\"a}t Bielefeld, Postfach 10 01 31, 33501 Bielefeld, Germany}
	\email{fgrube@math.uni-bielefeld.de}
	
	\makeatletter
	\@namedef{subjclassname@2020}{%
		\textup{2020} Mathematics Subject Classification}
	\makeatother
	
	

	\begin{abstract} 
		We study boundary regularity for the inhomogeneous Dirichlet problem for $2s$-stable operators in generalized H{\"o}lder spaces. Moreover, we provide explicit counterexamples that showcase the sharpness of our results. Our approach directly addresses the inhomogeneous Dirichlet problem, rather than subtracting an appropriate extension of the exterior data. Even for the fractional Laplacian, our result is new. 
	\end{abstract}

	\hypersetup{pageanchor=false}
	\maketitle
	\hypersetup{pageanchor=true}
		
	\section{Introduction}\label{sec:intro}
	
	The class of nondegenerate $2s$-stable integro-differential operators is a natural generalization of the fractional Laplacian. Given a finite, symmetric measure $\mu$ on the unit sphere $S^{d-1}$, i.e., $\mu(S^{d-1})\le \Lambda<\infty$ and $\mu(A)=\mu(-A)$, and an order $s\in (0,1)$, an operator in this class takes the form
	\begin{equation}\label{eq:stable_op}
		A_\mu^su(x):= (1-s)\pv \int\limits_{0}^\infty\int\limits_{S^{d-1}} \frac{u(x)-u(x+r\theta)}{r^{1+2s}} \mu(d\theta)\d r. 
	\end{equation}
	
	We always assume that $\mu$ is nondegenerate in the sense that 
	\begin{equation}\label{eq:nondegenerate}
		\inf\limits_{\xi\in S^{d-1}}\int\limits_{S^{d-1}} \abs{\theta\cdot \xi}^{2s}\mu(\d \theta)\ge \lambda.
	\end{equation}
	If we choose $\mu$ to be the uniform distribution on the sphere, then the resulting operator is the fractional Laplacian up to a harmless constant multiple. 
	
	These operators play a crucial role both in the field of partial differential equations and stochastic processes. They are the generators of nondegenerate $2s$-stable L{\'e}vy processes and generalize the fractional Laplacian. Being translation invariant, the nondegeneracy \eqref{eq:nondegenerate} and the finiteness of the measure $\mu$ yields the comparability of the Fourier symbol of $A_\mu^s$ to $\abs{\xi}^{2s}$.\smallskip
	
	Dirichlet problems involving $2s$-stable operators have received considerable attention in the past twenty years. An inhomogeneous Dirichlet problem may be stated as follows: Given a bounded open set $\Omega$ and an exterior datum $g:\Omega^c\to \R$, it is a natural question to ask for the regularity of a solution $u$ to
	\begin{equation}\label{eq:main}
		\begin{split}
			A_\mu^s u &= 0 \text{ in }\Omega\cap B_2,\\
			u&=g \text{ on }\Omega^c\cap B_2.
		\end{split}
	\end{equation}
	For the homogeneous Dirichlet problem, i.e., $g=0$, this question was the subject of many articles in the past 20 years. Beginning with the contributions from the field of potential theory for the fractional Laplacian \cite{Bog97, Kul97, ChSo98}, sharp global regularity of $u$ was discovered to be $C^{s}$, see \cite{RoSe14} and \cite{Gru14}. Note that this result requires some regularity of the boundary $\partial\Omega$. 
	
	For the inhomogeneous problem, however, much less is known in terms of sharp Hölder regularity. In the book \cite[Section 2.6.7]{RoFe24}, the problem was studied for $g\in C^{s+\eps}$ and $g\in C^{s-\eps}$. In the case $g\in C^{s+\eps}$, the authors found that the sharp global Hölder regularity $C^s$ remains true for solutions to \eqref{eq:main}. This is achieved by rewriting the problem as a homogeneous one with right-hand side. At the same time, in the case $g\in C^{s-\eps}$, solutions belong only to $C^{s-\eps}$ in general. Inhomogeneous exterior data are treated similarly, i.e., by extension to $\R^d$ and subtraction from the equation, in the articles \cite{Gru14, Gru15, AbGr23}.\smallskip
	
	This raises the question of whether the solution $u$ to \eqref{eq:main} is in the class $C^s(\overline{\Omega\cap B_{1/2}})$ whenever $g$ belongs to $C^s(\Omega^c)$. This is, contrary to intuition, generally false, see \autoref{cex:counterexample_cs_regularity}. This observation sets the goal of this article. \smallskip
	
	\textbf{Goal:} In this article we aim to pinpoint an appropriate condition on $g$ that yields Hölder continuous solutions $u\in C^s(\overline{\Omega})$. Moreover, given an exterior datum $g$, we aim to provide the modulus of continuity of the solution $u$ up to the boundary.\smallskip
	
	Nonlocal problems often behave differently from their local counterparts. One observation to emphasize is the following. For a solution $u$ to \eqref{eq:main} to be continuous up to the boundary of $\Omega\cap B_1$, we do not need $g$ to be continuous in the interior of $\Omega^c$. The intuition is simply that the values of $g$ further away from the boundary $\partial\Omega$ have a smaller influence on the solution. Nevertheless, we do, of course, need $g$ to be continuous at the boundary. This leads to an adaptation of the Hölder spaces for the exterior data to these nonlocal Dirichlet problems. \smallskip
	 
	For any nondecreasing function $\omega:[0,\infty)\to [0,\infty)$, we say that a function $g:\Omega^c\to \R$ belongs to the class $C_\text{\normalfont ext}^\omega(\Omega^c)$ if it is bounded on $\partial\Omega$ and there exists a finite constant $C_g$ such that
	\begin{equation*}
		\abs{g(y)-g(z)}\le C_g\, \omega\big(\abs{z-y}+d_{y}+d_{z}\big)
	\end{equation*}
	for all $y,z\in \Omega^c$. Here and throughout this article, $d_x$ denotes the distance of $x\in \R^d$ to the boundary $\partial\Omega$. We denote the smallest such constant $C_g$ by $[g]_{C_\text{\normalfont ext}^\omega(\Omega^c)}$. The space $C_{\text{\normalfont ext}}^{\omega}(\Omega^c)$ is equipped with the norm $\norm{g}_{C_\text{\normalfont ext}^\omega(\Omega^c)}:= \norm{g}_{L^\infty(\partial\Omega)}+ [g]_{C_\text{\normalfont ext}^\omega(\Omega^c)}$. 
	
	This class of functions is well suited for the phenomenon described above. Away from the boundary $\partial\Omega$, functions in the class $C_{\text{\normalfont ext}}^{\omega}$ are merely locally bounded. But on the boundary, they exhibit $C^\omega(\partial\Omega)$ regularity.\smallskip
	 
	Our main result gives a precise modulus of continuity of the solution $u$ to \eqref{eq:main} depending on the modulus $\omega$ of the exterior datum $g$. 
	 
	\begin{theorem}\label{th:regularity}
		Let $s_0\in (0,1)$, $\Omega\subset \R^d$, $d \in \N$, be an open set, and let $\mu$ be a symmetric, finite (i.e., $\mu(S^{d-1})\le \Lambda$) and nondegenerate, see \eqref{eq:nondegenerate}, measure on the unit sphere $S^{d-1}$. We assume one of the following properties 
		\begin{enumerate}
			\item[(1)] $\Omega$ satisfies the exterior $2s_0$-$\cdini$-property at every boundary point $z\in \partial\Omega \cap B_2$,
			\item[(2)] $\frac{\d\mu}{\d x}\le \Lambda$ and $\Omega$ satisfies the ext.\ $\cdini$-property at every boundary point $z\in \partial\Omega \cap B_2$,
			\item[(3)] $s_0>1/2$ and $\Omega$ satisfies the ext.\ $\cdini$-property at every boundary point $z\in \partial\Omega \cap B_2$.
		\end{enumerate}
		Let $\omega:[0,\infty)\to [0,\infty)$ be a nondecreasing function satisfying $\omega(0)=0$. Then for any given exterior datum $g\in C_{\text{\normalfont ext}}^{\omega}(\Omega^c\cap \overline{B_2})$ and $s\in[s_0,1)$ any weak or continuous distributional solution $u$ to \eqref{eq:main} belongs to $C^{\sigma}(\overline{\Omega\cap B_{1/2}})$, where  
		\begin{equation*}
			\sigma(t):= t^s\Big(1+\int_{t}^1 \frac{\omega(r)}{r^{1+s}}\d r\Big).
		\end{equation*}
		Moreover, we find a constant $C$ that depends only on $d, \Omega$, $\lambda$, $\Lambda$, and $s_0$ such that
		\begin{equation}\label{eq:hoelder_estimate_general}
			\abs{u(x)-u(y)}\le C \sigma(\abs{x-y})\Big(  \norm{g}_{C_{\text{\normalfont ext}}^{\omega}(\Omega^c\cap \overline{B_2})} +\norm{\tail_\mu^s[u]}_{L^\infty(\Omega\cap B_1)}\Big)
		\end{equation}
		for any $x,y\in \overline{B_{1/2}\cap\Omega}$.
	\end{theorem}

	Here, the tail of the function $u$ with respect to $A_\mu^s$ is defined as
	\begin{equation}\label{eq:tail}
		\tail_\mu^s[u](y):= (1-s)\int\limits_{1/2}^\infty\int\limits_{S^{d-1}}\frac{\abs{u(y+t\theta 	)}}{t^{1+2s}}\mu(\d \theta)\d t.
	\end{equation}
	We provide the exact definitions of the boundary assumptions in \autoref{th:regularity} in \autoref{sec:prelims}, see \eqref{eq:ext_dini_prop}. In essence, they say that at each boundary point $z\in \partial\Omega\cap B_2$ we can attach a certain paraboloid $\cC_{\omega}$, see \eqref{eq:ext_paraboloid}, from the outside at $z$ as in the following figure.
	\begin{figure}[!ht]
		\centering
		\begin{tikzpicture}[y=1cm, x=1cm, yscale=0.7,xscale=0.7, inner sep=0pt, outer sep=0pt]
			
			\begin{scope}[rotate=72];
				
				\coordinate (Cusp) at (0,-3.5);  
				\coordinate (Left) at (-2, 1);   
				\coordinate (Right) at (2.5, 1.5);  
				
				\draw[draw=white!20!black, fill=black!18, line width=1pt]	(Cusp)
				.. controls (0, -1.5) and (-0.5, 0.5) .. (Left)
				.. controls (-1.5, 3) and (1.5, 3.5) .. (Right)
				.. controls (3.38, 0) and (0, -0.5) .. (Cusp);
				\coordinate (Q) at (-0.3, -1.1);
				\fill[black, opacity=1, line width=0.000cm] (Right) circle (0pt) node at (-0.44,3.08) {$\Omega^c$};
				\fill[black, opacity=1] (0,2.3) circle (0pt) node[below] {$\Omega$};
				\fill[black] (Q) circle (4pt) node at (0.1, -0.93) {$z$};
			\end{scope}	
			
			\coordinate (Q1) at (1.5, -1.5);
			\coordinate (Q2) at (0.4, -1.5);
			
			\draw[draw=white!20!black, fill=black!5, line width=1pt]	(Q)
			.. controls (1.2,-0.7) and (1.4, -1)  .. (Q1)
			.. controls (Q2) .. (Q2)
			.. controls (0.5,-1) and (0.75,-0.7) .. (Q);
			
			\fill[black] (Q) circle (4pt) node at (0.95, -0.2) {};
			\fill[black] (Q) circle (4pt) node at (1.05, -1.14) {$\mathcal{C}_\omega$};
		\end{tikzpicture}
		\caption{Illustration of the exterior $\cdini$-property.}
		\label{fig:dini}
	\end{figure}
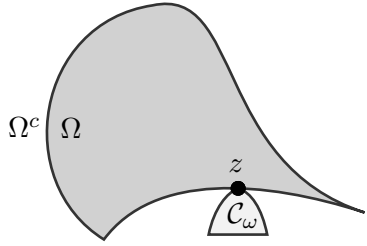\medskip
	
	Let us comment on \autoref{th:regularity} in interesting specific cases.\smallskip 
	
	\autoref{th:regularity} shows, in particular, that, if $\Omega$ is a bounded and $g$ belongs to $C^s(\Omega^c)$, then the solution $u$ is continuous and $C^s$ up to a multiplicative logarithmic blow-up.	The following corollary provides a criterion under which the solution belongs to $C^s(\overline{\Omega})$.
		\begin{corollary}\label{cor:regularity}
		Let $s, \mu, \Omega$, and $d$ be as in \autoref{th:regularity}. If the exterior datum $g:\Omega^c\to\R$ satisfies
		\begin{equation*}
			\abs{g(y)-g(z)}\le (\abs{y-z}+d_y+d_z)^s\,\iota(\abs{y-z}+d_y+d_z)
		\end{equation*} 
		for all $y,z\in \Omega^c$, where $\iota:[0,\infty)\to [0,\infty)$ is a nondecreasing function satisfying the Dini property $\int_0^1 \iota(t)/t \d t <\infty$, then any weak or continuous distributional solution $u$ to \eqref{eq:main} belongs to the Hölder space $C^s(\overline{\Omega\cap B_{1/2}})$.
	\end{corollary}

	\begin{remark}\		
		\begin{enumerate}
			\item[(i)] In the case that $\mu$ has a strictly positive, bounded density with respect to the surface measure on the sphere $S^{d-1}$, i.e., $\lambda<\mu(\theta)\le \Lambda$, the term $\norm{\tail_\mu^s[u]}_{L^\infty(B_{1})}$ is comparable to 
			\begin{equation*}
				\norm{u}_{L_{2s}^1(\R^d)}=(1-s) \int_{\R^d}\frac{\abs{u(x)}}{(1+\abs{x})^{d+2s}}\d x,
			\end{equation*} 
			where $L_{2s}^1(\R^d)$ is the typical tail space, see, e.g., \cite{RoFe24}.
			\item[(ii)] Due to condition \textit{(2)}, if we are only interested in the fractional Laplacian, then \autoref{th:regularity} holds true for open sets $\Omega\subset \R^d$ satisfying the exterior $\cdini$-property.
			\item[(iii)]The regularity assumptions on the domain $\Omega$ in \autoref{th:regularity} are due to the regularity results for the homogeneous Dirichlet problem \cite{Gru24}. 
		\end{enumerate} 
	\end{remark}	
	
	In the case of the Laplacian such results have been known for a long time. We refer to \cite[Section 1.2.4]{Maz18}, \cite[Theorem 1.2]{Tor25}, or \cite[Lemma 4]{Dya97} for the case of the ball. Note that \autoref{th:regularity} is robust in the limit $s\to 1$ and we recover the classical result for translation invariant operators $A\cdot D^2$. Indeed, \eqref{eq:hoelder_estimate_general} for $s=1$ is the sharp regularity estimate in the case of the Laplacian.
	
	\begin{remark}
		If we additionally assume that $\Omega$ is a Lipschitz domain, then instead of weak or continuous distributional solutions, we can consider distributional solutions $u\in L^1(\R^d, \nu_\mu^{s,\star})\cap L^1(\Omega\cap B_2, d_{x}^{-s})$ in \autoref{th:regularity}. This is due to the maximum principle obtained in \cite[Section 4.1]{Gru25}. For the definition of $\nu_\mu^{s,\star}$, we refer to \autoref{sec:prelims}.
	\end{remark}

	Through explicit examples, we showcase the sharpness of \autoref{cor:regularity}.
	\begin{counterexample}\label{cex:counterexample_cs_regularity}
		For every nondecreasing function $\iota:[0,\infty)\to [0,\infty)$ that is not Dini continuous in the origin, i.e., $\int_{0}^1\iota(t)/t \d t =+\infty$, there exists an exterior datum $g\in C_c^{\omega}(B_1(0)^c)$, $\omega(t)= t^s\iota(t)$, such that the weak solution to 
		\begin{equation}\label{eq:inhom_dirichlet_ball}
			\begin{split}
				(-\Delta)^s u &= 0 \text{ in } B_1(0),\\
				u&= g \text{ on }B_1(0)^c
			\end{split}
		\end{equation}
		does not belong to $C^s(\overline{B_1(0)})$. 
	\end{counterexample}
	
	In fact, we can construct a much more delicate explicit solution and prove a precise lower bound in the following theorem. This offers insight into the sharpness of \autoref{th:regularity}.
	
	\begin{theorem}\label{th:counterexample_general}
		Let $s_0 \in (0,1)$ and $d\ge 2$. There exists a positive constant $C$ such that for any nondecreasing function $\omega:[0,\infty)\to [0,\infty)$ and any $s\in [s_0,1)$, we find an exterior datum $g\in C_c^{\omega}(B_1(0)^c)$ such that the weak solution $u$ to \eqref{eq:inhom_dirichlet_ball} satisfies
		\begin{equation*}
			\abs{u(te_1)-g(e_1)}\ge C (1-t)^s \int\limits_{1-t}^1 \frac{\omega(r)}{r^{1+s}}\d r
		\end{equation*}
		for $0\le t\le 1$.
	\end{theorem}
	
	This theorem shows, in particular, that in the class of function $C_c^s(\Omega^c)$ there exists an exterior datum $g$ such that the corresponding solution $u$ to \eqref{eq:inhom_dirichlet_ball} is not more regular than $\log-C^s$-continuous.
	
	We provide the corresponding result in one spatial dimension in \autoref{prop:counterexample_general_d1}.
	
	\begin{remark}\label{rem:uniformity}
		We can pick the exterior data in both \autoref{cex:counterexample_cs_regularity} and \autoref{th:counterexample_general} such that $g$ is smooth along any ray starting at zero, i.e., $t\mapsto g(t\theta)\in C_c^\infty([1,\infty))$ for any vector $\theta$ on the sphere $S^{d-1}$. This is even true uniformly in the parameter $\theta\in S^{d-1}$.
	\end{remark}
	
	But, there exist exterior data in the class $C_c^s(\Omega^c)$, that are not any more regular on the  boundary than $C^s(\partial\Omega)$, such that the solution $u$ to \eqref{eq:main} is indeed Hölder continuous of order $s$, see \autoref{ex:countercounterexample}. The key difference is that these exterior data are sign-changing and the corresponding solutions benefit from cancellation effects. \medskip
		
	\subsection{Key idea}
	In contrast to previous works on inhomogeneous fractional Dirichlet problems, we do not subtract an extension of the exterior data from the equation in the proof of \autoref{th:regularity}. Instead, we work with the solution to the inhomogeneous problem directly. This allows us to immediately derive the modulus of continuity of the solution $u$ from that of the exterior datum. The key ingredient is an interior-to-boundary estimate, which essentially states that for $x\in \Omega$ and $z\in \partial\Omega$
	\begin{equation*}
		\abs{u(x)-g(z)}\le \int_0^\infty v_t(x)\d \omega(t).
	\end{equation*}
	Here, as in \autoref{th:regularity}, the nondecreasing function $\omega$ is the modulus of continuity of $g$ and $v_t$ is the solution to the problem
	\begin{equation}\label{eq:hom_sol}
		\begin{split}
			A_\mu^s v_t &= 0 \text{ in }\Omega,\\
			v_t&= 0 \text{ on }\Omega^c\cap B_t(z),\\
			v_t&=1 \text{ on }\Omega^c\setminus B_t(z). 
		\end{split}
	\end{equation}
	With this formulation, we can apply regularity theory for the homogeneous Dirichlet problem to derive a global bound on $v_t$ for all $t>0$. 
	
	\subsection{Related results}
	In the past decade, boundary regularity results for problems involving nonlocal operators have enjoyed considerable attention. Important contributions to this line of research are the articles \cite{RoSe14} and \cite{Gru14}. In \cite{RoSe14}, the authors establish the sharp $C^s$ boundary regularity of the Dirichlet problem for the fractional Laplacian in bounded $C^{1,1}$-domains. Moreover, higher order expansions of the quotient $u(x)/d(x)^s$ and Hölder continuity thereof are derived. This is done with the aid of explicit barrier functions. In \cite{Gru14} and \cite{Gru15}, the fractional Laplacian and fractional elliptic pseudo-differential operators and related problems in smooth domains are studied. The author treats various boundary conditions both of local and nonlocal nature and derives a range of regularity results. Among other results, the sharp boundary regularity of solutions to \eqref{eq:main} for the fractional Laplacian is studied with exterior data $g\in C_c^{2s}(\Omega^c)$. In these works, to handle nonzero exterior data, these data are extended to the interior and subtracted to apply results for the homogeneous Dirichlet problem. The class of nondegenerate $2s$-stable operators in bounded $C^{1,1}$-domains is studied in \cite{RoSe16a}. Firstly, interior regularity is derived and, thereafter, boundary regularity of the solution $u$ (and $u(x)/d_x^s$) to the homogeneous exterior value problem $A_\mu^su=f$ in $\Omega$ and $u=0$ on $\Omega^c$. The sharpness of these results is demonstrated by two counterexamples. The sharp $C^s$-boundary regularity for such solutions fails in general in $C^1$-domains. Nevertheless, regularity up to the boundary can be obtained. This was achieved in \cite{RoSe17}. The line of research initiated in the articles \cite{Gru14} and \cite{Gru15} is generalized to the case of non-smooth domains in the work \cite{AbGr23}. A vast amount of regularity results are obtained in the book \cite{RoFe24}. In particular, the previously mentioned boundary regularity result for solutions to fractional Dirichlet problems with nonzero exterior data of the class $C^{s+\eps}(\Omega^c)$ and $C^{s-\eps}(\Omega^c)$ are contained in Section 2.6.7. This line of research has pursued two longstanding goals: generalizing the regularity results in terms of a) the class of nonlocal integro-differential operators and b) the regularity assumptions on the domain. In \cite{RoWe24}, the sharp global regularity $C^s$ of solutions to equations involving translation invariant nonlocal operators of the type
	\begin{equation*}
		L_ku(x):= \pv \int_{\R^d}(u(x)-u(y))k(x-y)\d y
	\end{equation*}
	is derived for nonhomogeneous kernels $k$, which are symmetric and pointwise comparable to that of the fractional Laplacian (i.e., $k(h)\asymp \abs{h}^{-d-2s}$). The main obstruction to achieving this result is the construction of an appropriate barrier, which in previous results heavily relied on homogeneity of the kernels of the nonlocal operators. With the aid of a nonlocal free boundary problem, the authors of \cite{RoWe24} construct a new barrier function. This result is generalized in \cite{KiWe24} to $x$-dependent kernels that satisfy a certain Hölder continuity along the diagonal. This is done by freezing the coefficients. Another generalization in this direction is contained in the article \cite{BKK25}. Therein, the coefficients of the kernels are assumed to be of vanishing mean oscillation or Dini continuous. In \cite{Gru24}, boundary regularity for solutions to exterior value problems involving nondegenerate $2s$-stable operators is generalized to bounded open sets that satisfy the exterior $\cdini$-property (resp.\ $2s$-$\cdini$ property), see \autoref{sec:prelims}. Almost optimal regularity, i.e., $C^{s-\eps}$, in Reifenberg flat domains is considered in \cite{Pra25}. Many of the previously mentioned articles also contain Hopf-type boundary estimates for nonnegative supersolutions. Besides the Dirichlet boundary / exterior condition, other conditions like nonlocal and local variants of the Neumann condition are studied, see, e.g., \cite{Montefusco_neumann}, \cite{Gru14}, \cite{Barles_Neumann}, \cite{Barles_2_Neumann}, \cite{Abatangelo_Neumann}, \cite{CiCo20}, \cite{AFR23}, and \cite{FoKa22}. \smallskip
	
	Nonlocal operators whose integration domain is restricted to the set on which the equation holds, i.e., $\Omega$, are typically called regional nonlocal operators. Boundary regularity results for the regional fractional Laplacian of order $2s$ and related operators are obtained in \cite{ChKi02}, \cite{Che18}, and \cite{Fal22}. The optimal Hölder regularity in this case is $C^{2s-1}(\overline{\Omega})$. This leads to the question of how the boundary regularity of solutions to the Dirichlet problem changes as the integration domain of the underlying nonlocal operators is further restricted, e.g., to balls whose radii depend on the distance to the boundary $B_{d_x/2}(x)$. These questions are studied in the articles \cite{Cha23} and \cite{Svi25}. \smallskip
	
	An important contribution to the regularity of nonlinear nonlocal integro-differential operators is the article \cite{CaSi09}. Boundary regularity for solutions to Dirichlet problems involving such nonlinear nonlocal operators has been a topic of research as well. We only mention a few results in this direction and refer to the references in the survey \cite{Ian24}. Fully nonlinear nonlocal integro-differential equations are treated in \cite{RoSe16} and results on nonlinear nonlocal operators like the fractional $p$-Laplacian are contained in \cite{IMS16}, \cite{IaMo24}, and \cite{BKS25}.
	
	\subsection{Outline} The article is structured as follows. In \autoref{sec:prelims}, we establish the necessary mathematical framework, introducing the notation, function spaces, geometric properties, and solution concepts used throughout the paper. The core of the paper is \autoref{sec:regularity}, which is dedicated to the proof of our main result. In the final section, explicit solutions on the ball are constructed that demonstrate the sharpness of \autoref{th:regularity} and \autoref{cor:regularity}. 
	
	\subsection*{Acknowledgments} I am grateful to Waldemar Schefer for valuable discussions and thank Moritz Kassmann and Philipp Svinger for comments on the final manuscript. I would also like to thank Xiaochuan Tian for bringing the problem to my attention. Financial support by the German Research Foundation (DFG - Project number 541771122) is gratefully acknowledged. 
	
	\section{Preliminaries}\label{sec:prelims}
	We briefly introduce the notation, conventions, boundary assumptions, and function spaces used throughout this article. 
	
	The unit sphere in the $d$-dimensional Euclidean space is denoted by $S^{d-1}$. A ball with radius $r$ and with center $x\in \R^d$ is written as $B_r(x)$. At times, we omit the center point and simply write $B_r$. By $\Gamma(\cdot)$ we denote Euler's gamma function. For a vector $x\in \R^d$, we use the convention $x=(x',x_d)$ where $x'\in \R^{d-1}$ and $x_d\in \R$. The tail of a function $u$, corresponding to an operator $A_\mu^s$ of the form \eqref{eq:stable_op}, is defined in \eqref{eq:tail}. For a measure $\mu$ on $\R^d$ or on $S^{d-1}$ we denote by $\frac{\d \mu}{\d x}$ the Radon–Nikodym derivative with respect to the Lebesgue measure or the surface measure on the unit sphere. For a real-valued function $f:\R\to \R$ and a nondecreasing function $g:\R\to \R$, let $\int_{\R} f(t)\d g(t)$ be the Riemann–Stieltjes integral of $f$ with respect to $g$.
		
	Throughout this article, we always use a bounded open set $\Omega$, for which we introduce the notation $x\mapsto d_x:=\min\{ \abs{x-y}\mid y\in \partial\Omega \}$ for the distance from $x$ to the boundary of $\Omega$. \smallskip
	
	The space of Hölder continuous functions on $\overline{\Omega}$ is written as $C^{\alpha}(\overline\Omega)$. For a nondecreasing function $\omega:[0,\infty)\to [0,\infty)$, the space of functions $C^{\omega}(\overline{\Omega})$ contains all bounded functions $u:\overline{\Omega}\to \R$ such that 
	\begin{equation*}
		\norm{u}_{C^{\omega}(\overline{\Omega})}:= \norm{u}_{L^\infty(\overline{\Omega})}+ \sup\limits_{x,y\in \overline{\Omega}} \frac{\abs{u(x)-u(y)}}{\omega(\abs{x-y})}
	\end{equation*} 
	is finite. In a similar fashion, we define the class of functions $C_\text{\normalfont ext}^\omega(\Omega^c)$ as all functions $g:\Omega^c\to \R$ that are bounded on $\partial\Omega$ and such that 
	\begin{equation*}
		[g]_{C_\text{\normalfont ext}^\omega(\Omega^c)}:= \sup\limits_{y,z\in\Omega^c}\frac{\abs{g(y)-g(z)}}{\omega\big(\abs{z-y}+d_{y}+d_{z}\big)}
	\end{equation*}
	is finite. The space $C_\text{\normalfont ext}^\omega(\Omega^c)$ is equipped with the norm $\norm{g}_{C_\text{\normalfont ext}^\omega(\Omega^c)}:=\norm{g}_{L^\infty(\partial\Omega)}+ [g]_{C_\text{\normalfont ext}^\omega(\Omega^c)}$. The space $C_\text{\normalfont ext}^\omega(\Omega^c\cap \overline{B_2})$ and the corresponding norm is defined in a similar fashion but $d_x$ remains the distance to $\partial\Omega$.
	
	We say that $\Omega$ satisfies the exterior $\cdini$-property at a boundary point $z\in \partial\Omega$ if we find a rotation $R_z$, a radius $\rho>0$, and a nondecreasing function $\omega:[0,\infty)\to [0,\infty)$ satisfying $\omega(0)=0$ and the Dini property
	\begin{equation}\label{eq:dini_prop}
		\int_0^1 \frac{\omega(t)}{t}\d t <\infty
	\end{equation}
	such that 
	\begin{equation}\label{eq:ext_dini_prop}
		R_z\big({-z}+\Omega\big) \subset \mathcal{C}_\omega^c
	\end{equation}
	where $\mathcal{C}_\omega$ is a finite paraboloid with apex at zero and slope governed by the function $t\mapsto t\omega(t)$
	\begin{equation}\label{eq:ext_paraboloid}
		\mathcal{C}_\omega:=\{ (x',x_d)\in \R^{d}\mid -\rho<x_d<-\abs{x'}\omega(\abs{x'}) \}
	\end{equation}
	This property is illustrated in \autoref{fig:dini}.
	
	We say that $\Omega$ satisfies the exterior $2s$-$\cdini$-property at a boundary point $z\in \partial\Omega$ if it satisfies the ext.\ $\cdini$-property but instead of \eqref{eq:dini_prop} the term $\int_0^1 \omega(t)^{2s}/t \d t$ is finite.
	
	Whenever we say that the exterior $\cdini$-property holds for all boundary points $z\in \partial\Omega\cap B_2$, the radius $\rho$ and $\omega$ are uniform.
	
	We say that $u$ is a weak solution to the problem
	\begin{equation}\label{eq:main_with_RHS}
		\begin{split}
			A_\mu^su&= f \text{ in }\Omega,\\
			u&= g \text{ on }\Omega^c,
		\end{split}
	\end{equation}
	if $u\in L^2(\Omega)$, $\cE_\mu^s(u,u)$ is finite, where $\cE_\mu^s(u,v)$ is defined as
	\begin{equation*}
		\frac{1-s}{2}\int\limits_{\R^d}\int\limits_{0}^\infty\int\limits_{S^{d-1}} \1_{(\Omega^c\times\Omega^c)^c}(x,x+r\theta)\frac{(u(x)-u(x+r\theta))(v(x)-v(x+r\theta))}{r^{1+2s}}\mu(\theta)\d r \d x,
	\end{equation*}
	$u=g$ on $\Omega^c$, and $$\cE_\mu^s(u,v)= \int_{\Omega} f(x)v(x)\d x$$ for any $v\in L^2(\Omega)$ such that $\cE_\mu^s(v,v)<\infty$ and $v=0$ on $\Omega^c$. 
	
	We instead say that $u$ is a continuous distributional solution to \eqref{eq:main_with_RHS} if $u$ is continuous on $\overline{\Omega}$, belongs to $L^1(\Omega)\cap L^1(\R^d,\nu_\mu^{s, \star})$, and satisfies 
	\begin{equation}\label{eq:distr_s_harmonic}
		\int\limits_{\R^d} u(x)A_\mu^sv(x)\d x = \int\limits_{\Omega} f(x)v(x)\d x
	\end{equation}
	for all $v\in C_c^\infty(\Omega)$. Here, $\nu_\mu^{s,\star}$ is a weight adapted to $A_\mu^s$ defined as in \cite[(2.1.4)]{Gru25} via
	\begin{equation*}
		\nu_\mu^{s,\star}(x):= (1-s)\int\limits_{\R} \int\limits_{S^{d-1}} \1_{\Omega}(x+r\theta)(1+\abs{r})^{-1-2s}\mu(\d \theta)\d r.
	\end{equation*}
	The space $L^1(\R^d,\nu_\mu^{s,\star})$ is sufficient for the distributional formulation \eqref{eq:distr_s_harmonic} to be well defined. 
	
	\section{Boundary regularity}\label{sec:regularity}
	
	The dominating contribution in the regularity result \autoref{cor:regularity} is the regularity in normal direction, i.e., by bounding the difference $\abs{u(x)-g(z)}$ for $x\in \Omega$ and $z\in \partial\Omega$. Before we proceed, we need the following auxiliary result. 
	
	\begin{lemma}\label{lem:asymptotically_hom}
		Let $s_0\in (0,1)$, $\Omega\subset \R^d$, $d \in \N$, be an open set, $z\in \partial\Omega$, and let $\mu$ be a symmetric, finite, and nondegenerate, see \eqref{eq:nondegenerate}, measure on the unit sphere. We assume one of the following properties 
		\begin{enumerate}
			\item[(1)] $\Omega$ satisfies the exterior $2s_0$-$\cdini$-property at $z$,
			\item[(2)] $\frac{\d \mu}{\d x}\le \Lambda$ and $\Omega$ satisfies the ext.\ $\cdini$-property at $z$,
			\item[(3)] $s_0>1/2$ and $\Omega$ satisfies the ext.\ $\cdini$-property at every boundary point $z$.
		\end{enumerate}
		Let $v_t$ be a solution to the problem \eqref{eq:hom_sol} for $t>0$. There exists a constant $C=C(s_0, d,\Omega, \lambda,\Lambda)$ independent of $s,t$, and $z$ such that 
		\begin{equation*}
				v_t(x)\le C \big( \frac{\abs{x-z}}{t}\wedge 1\big)^s \text{ for }t\le 1,\quad
				v_t(x)\le (1-s)C  \frac{\abs{x-z}^s}{t^{2s}} \text{ for }t> 1.
		\end{equation*}
	\end{lemma}
	\begin{proof}
		Since the problem is translation invariant, we assume $z=0$ without loss of generality. The estimate in the case $t>1$ is an easy corollary from the boundary regularity estimate \cite[Proposition 6.1 or 6.3]{Gru24}  and the observation that $0\ge A_\mu^s\1_{B_t(z)^c}(x)\ge -c(1-s)(t-\diam(\Omega))^{-2s}$ for $t\ge 1+\diam(\Omega)$ and $x\in \Omega$. We turn our attention to the case when $t\le 1$. For $t=1$, the result is an immediate consequence of \cite[Proposition 6.1 or 6.3]{Gru24}. Since the open set $\frac{1}{t}\Omega$ satisfies the exterior $\cdini$-property at $z=0$ with a modulus independent of $t$, the result in the case $t=1$ together with the function $\tilde{v}_t(y):= v_t(ty)$ yields 
		\begin{equation*}
			\abs{v_t(x)}=\abs{\tilde{v}_t(x/t)} \le c_2 \abs{x/t-z}^s= c_2 \Big(\frac{\abs{x-z}}{t}\Big)^s.
		\end{equation*}
		Finally, the comparison principle yields $0\le v_t\le 1$ which completes the proof. 
	\end{proof}
	
	The following proposition is a key ingredient in the proof of \autoref{th:regularity}. It provides the regularity of $u$ in non-tangential directions, connecting $u$ to its exterior datum. In the proof of this result, we adapt \cite[Theorem 1.2.5]{Maz18} for the Laplacian to the nonlocal setup. 
	
	\begin{proposition}\label{prop:regularity_with_modulus}
		Let $\Omega\subset \R^d, \mu, \lambda, \Lambda, s_0$, and $s$ be as in \autoref{th:regularity}, $g:\Omega^c\to \R$ be a function such that $r^{-2s}\max_{B_r(0)}\abs{g}$ converges to zero as $r\to +\infty$, and $u$ be a weak or continuous distributional solution to 
		\begin{align*}
			A_\mu^s u&= 0\text{ in }\Omega,\\
			u&=g \text{ on }\Omega^c.
		\end{align*}
		Then, we find a constant $C$ depending only on $s_0, d, \Omega, \lambda$, and $\Lambda$ such that for any $x\in \Omega$ and $z\in \Omega^c$
		\begin{equation*}
				\abs{u(x)-g(z)} \le C \abs{x-z}^s \Bigg(\,\int\limits_{\abs{x-z}}^1 \frac{\max\limits_{\abs{z-w}\le t}\abs{g(z)-g(w)}}{t^{1+s}}\d t+(1-s)\int\limits_{1}^\infty \frac{\max\limits_{\abs{z-w}\le t}\abs{g(z)-g(w)}}{t^{1+2s}}\d t\Bigg).
		\end{equation*}
	\end{proposition}	
	
	\begin{proof}[{Proof of \autoref{prop:regularity_with_modulus}}]
		We define $\xi_{z}(t):= \max\{ \abs{g(z)-g(w)}: \abs{z-w}\le t \}$. Note that $\xi_z$ is nondecreasing. The Poisson representation yields
		\begin{align*}
			\abs{u(x)-g(z)}&= |\int\limits_{\Omega^c} \big(g(y)-g(z)\big)P_s(x,y)\d y|\le \int\limits_{\Omega^c} \xi_z(\abs{z-y})P_s(x,y)\d y\\
			&= \int\limits_{\Omega^c}\int\limits_{0}^{\abs{z-y}} \d\xi_z(t)P_s(x,y)\d y= \int\limits_{0}^{\infty}  \int\limits_{\Omega^c} \1_{B_t(z)^c}(y) P_s(x,y)\d y\d\xi_z(t).
		\end{align*}
		Using the notation in \autoref{lem:asymptotically_hom}, we find a constant $c_1$ such that $\abs{u(x)-g(z)}$ is not larger than
		\begin{equation*}
			\int\limits_{0}^{\infty}  v_t(x)\d\xi_z(t)\le c_1\int\limits_0^{\abs{x-z}}  \d\xi_z(t)+c_1 \int\limits_{\abs{x-z}}^{1}  \frac{\abs{x-z}^s}{t^s}\d\xi_z(t)+ c_1 \abs{x-z}^s\int\limits_{1}^\infty t^{-2s}\d \xi_z(t).
		\end{equation*}
		The first term in the previous line simply equals $c_1\xi_z(\abs{x-z})$. The second term, after integration by parts, equals 
		\begin{equation*}
			c_1\int_{\abs{x-z}}^{1}  \frac{\abs{x-z}^s}{t^s}\d\xi_z(t)= c_1\xi_z(1)\abs{x-z}^s- c_1\xi_z(\abs{x-z})+c_1 s \abs{x-z}^s \int_{\abs{x-z}}^1 \frac{\xi_z(t)}{t^{1+s}}\d t.
		\end{equation*}
		Under the assumption that $\xi_z(t)/t^{2s}\to 0$ as $t\to \infty$, the third term becomes $c_1 \abs{x-z}^s \int_{1}^\infty \frac{\xi_z(t)}{t^{1+2s}}\d t$.
	\end{proof}
	
	In the final result of \autoref{th:regularity}, we will combine \autoref{prop:regularity_with_modulus} with known interior $C^s$-regularity in sufficiently small balls. To be able to treat general moduli of continuity in \autoref{th:regularity}, we need a scaling-sensitive embedding between generalized Hölder spaces. The next lemma provides this.
	
	\begin{lemma}[A generalized Hölder embedding]\label{lem:hoelder}
		Let $\kappa:[0,\infty)\to (0,\infty)$ be a nonincreasing function. We define $\iota(t):= t^s\kappa(t)$, $s\in (0,1)$, and $\eps\in (0,1/2]$. For any $r>0$ and any $u\in C^s(B_r)$ we find 
		\begin{equation*}
			[u]_{C^{\iota}(B_{\eps r})}\le \frac{1}{\kappa(r)} [u]_{C^s(B_{\eps r})}.
		\end{equation*}
	\end{lemma}
	\begin{proof}
		We prove the result in two steps. First, we consider the case $r=1$ and define an auxiliary function
		\begin{equation*}
			\tilde{\iota}(t):= t^s\inf\limits_{r>0} \frac{\kappa(rt)}{\kappa(r)}.
		\end{equation*}
		Note that $\tilde{\iota}(t)\ge t^s$ for any $0<t\le 1$ since $\kappa$ is nonincreasing. Thus, we have shown that 
		\begin{equation}\label{eq:gen_hoelder_unit}
			[u]_{C^{\tilde{\iota}}(B_{\eps})}\le [u]_{C^s(B_{\eps})}.
		\end{equation}
		
		Now, we come to the second part of the proof. By applying the result from the $B_{\eps}$ ball \eqref{eq:gen_hoelder_unit} to the function $\tilde{u}(z)=u(rz)$ for $r>0$, one obtains for any $x,y\in B_{\eps r}$
		\begin{equation*}
			\frac{\abs{u(x)-u(y)}}{\tilde{\iota}(\abs{x-y}/r)}\le r^s[u]_{C^s(B_{\eps r})}.
		\end{equation*}
		Since 
		\begin{equation*}
			\tilde{\iota}(\abs{x-y}/r) \le \frac{\abs{x-y}^s}{r^s} \frac{\kappa(r \frac{\abs{x-y}}{r} )}{\kappa(r)} =  \frac{\iota(\abs{x-y})}{\iota(r)} 
		\end{equation*}
		by definition, the result follows directly.
	\end{proof}
	
	The following two lemmata are of technical nature. 
	
	\begin{lemma}\label{lem:almost_increasing}
		The function $\sigma$ as defined in \autoref{th:regularity} is almost nondecreasing, i.e., there exists a constant $c$ such that 
		\begin{equation*}
			\sigma(t_1)\le c\sigma(t_2)
		\end{equation*}
		for all $0\le t_1\le t_2$. Moreover, it satisfies 
		\begin{equation*}
			\sigma(at)\le a \sigma(t)
		\end{equation*}
		for all $a\ge 1$ and $t>0$. 
	\end{lemma}
	\begin{proof}
		Clearly, $\sigma$ is increasing on $[1,\infty)$. Thus, it remains to study $\sigma$ on $[0,1]$. \smallskip	
		
		We take the derivative of $\sigma$ with respect to $t$ at $t\in (0,1)$. 
		\begin{equation*}
			\partial_t \sigma(t)= st^{s-1}(1+\int_t^1 \omega(r)/r^{1+s}\d r)- \omega(t)/t\ge st^{s-1}+t^{-1}(\int_t^1 s\frac{\omega(r)}{r}-\frac{\omega(t)}{1-t}\d r).
		\end{equation*}
		This is nonnegative if $t\le s/(1+s)$ since $\omega$ is nondecreasing and, thus, $\sigma$ is increasing on $[0,s/(s+1)]$. On the interval $[s/(s+1),1]$, the function $\sigma$ is comparable to $1$ and, thus, the result follows. \smallskip
		
		The second part is a direct consequence of the definition of $\sigma$. Let $a\ge 1$ and $t>0$, then 
		\begin{equation*}
			\sigma(at)= a^st^s\big(1+ \int_{at}^1\frac{\omega(r)}{r^{1+s}}\d r\big)\le a t^s\big(1+ \int_{t}^1\frac{\omega(r)}{r^{1+s}}\d r\big)\le a\sigma(t).
		\end{equation*}
	\end{proof}
	
	\begin{lemma}\label{lem:sigma_vs_omega}
		Let $\sigma$ and $\omega$ be as in \autoref{th:regularity}, then there exists a constant $C=C(\omega(2))$ such that 
		\begin{equation*}
			\omega(t)\le C\sigma(t)
		\end{equation*}
		for all $0\le t\le 2$. 
	\end{lemma}
	\begin{proof}
		We set $C:= 2\omega(2)\vee 2$. If $t<(1-s/2)^{1/s}$, then the result follows from
		\begin{align*}
			\sigma(t)\ge t^s \int_t^1 \frac{\omega(r)}{r^{1+s}}\d r\ge \omega(t)t^s \frac{t^{-s}-1}{s}=  \omega(t)\frac{1-t^s}{s}\ge \frac{1}{2}\omega(t).
		\end{align*}
		If instead $(1-s/2)^{1/s}\le t\le 2$, then 
		\begin{equation*}
			\omega(t)\le \omega(2)\le C ((1-s/2)^{1/s})^s\le C t^s\le C\sigma(t).
		\end{equation*}
	\end{proof}
	
	With all necessary ingredients at hand, we proceed to prove our main result.

	\begin{proof}[{Proof of \autoref{th:regularity}}]
		Without loss of generality, we assume that the ball $B_2$ is centered at a boundary point $\tilde{x}\in\partial\Omega$. Let $\eta$ be a smooth, nonnegative cutoff function such that $\eta\le 1$, $\eta=1$ in $B_{3/2}$, and $\eta=0$ on $B_2^c$. We set $v:= u\eta$. This function satisfies in the weak (respectively distributional sense) $A_\mu^sv(x)=f(x)$ for $x\in \Omega\cap B_1$ where $f(x)=-A_\mu^s[u(1-\eta)](x)$ is bounded by
		\begin{equation*}
			 \abs{f(x)}\le |\int\limits_{0}^\infty\int\limits_{S^{d-1}}\frac{u(x+r\theta)}{r^{1+2s}}(1-\eta(x+r\theta))\mu(\d \theta)\d r|\le \sup_{y\in \Omega\cap B_1} \int\limits_{1/2}^\infty\int\limits_{S^{d-1}}\frac{\abs{u(y+r\theta)}}{r^{1+2s}}\mu(\d \theta)\d r.
		\end{equation*}
		Leveraging the known boundary regularity theory for the homogeneous Dirichlet problem, see \cite{Gru24}, the weak solution $v_1$ to 
		\begin{equation*}
			A_\mu^s v_1=f \text{ in }\Omega\cap B_1,\quad
			v_1=0 \text{ on }\big(\Omega^c\cap B_1\big)\cup B_1^c
		\end{equation*}
		belongs to $C^s(\overline{B_{1}})$ with the bound for any $y\in \overline{B_{2/3}}\cap \overline{\Omega}$
		\begin{equation}\label{eq:v_1-boundary_estimate}
			\abs{v_1(y)-0}\le c_1 d_y^s \norm{f}_{L^\infty(\Omega\cap B_1)}\le c_1 d_y^s\sup_{y\in \Omega\cap B_1} \int\limits_{1/2}^\infty\int\limits_{S^{d-1}}\frac{\abs{u(y+r\theta)}}{r^{1+2s}}\mu(\d \theta)\d r.
		\end{equation}
		Note that the function $v_2:=u-v_1$ satisfies in the weak sense (or $v_2$ is continuous and satisfies in the distributional sense)
		\begin{equation*}
			\begin{split}
				A_\mu^sv_2 &= 0\,\,\; \text{ in }\Omega\cap B_1,\\
				v_2 &= \eta g \text{ on }\Omega^c \cap B_1,\\
				v_2&= 0 \,\,\;\text{ on }B_2^c.
			\end{split}
		\end{equation*}
		Due to the boundary estimate \autoref{prop:regularity_with_modulus}, we find for any $y\in B_{d_x/2}(x)$
		\begin{equation}\label{eq:v_2-boundary_estimate}
			\begin{split}
				\abs{v_2(y)-g(x_0)}&\le c_2 [g\eta]_{C^{\omega}_{\text{\normalfont ext} } (\Omega^c)}\abs{y-x_0}^s\Big( \int\limits_{\abs{y-x_0}}^2 \frac{\omega(t)}{t^{1+s}}\d t +\int\limits_{2}^\infty \frac{\omega(2)}{t^{1+2s}}\d t\Big)\\
				&\le 2^sc_3\norm{g}_{C^{\omega}_{\text{\normalfont ext} } (\Omega^c\cap \overline{B_2})}d_x^s\big( 1+\int\limits_{d_x/2}^{1/2} \frac{\omega(t)}{t^{1+s}}\d t \big)
			\le 4c_3\norm{g}_{C^{\omega}_{\text{\normalfont ext} } (\Omega^c\cap \overline{B_2})}\sigma(d_x).
			\end{split}
		\end{equation}		
		In the last inequality, we used the change of variables $r=2t$ and the monotonicity of $\omega$. By standard interior regularity estimates \cite{RoSe14, RoSe16a,RoFe24,Gru24} and by scaling, we find a constant $c_2$ such that for $x\in \Omega\cap B_{1/2}$ the term $[u]_{C^s(B_{d_x/4}(x))}$ is bounded from above by 
		\begin{equation}\label{eq:interior}
			\begin{split}
				[u]_{C^s(B_{d_x/4}(x))}=[v]_{C^s(B_{d_x/4}(x))}&\le c_2\Big( d_x^s\norm{f}_{L^\infty(B_{d_x/2}(x))}+d_x^{-s}\norm{v}_{L^\infty(B_{d_x/2}(x))}\\
				& + (1-s)d_x^{-s}\sup_{y\in B_{d_x/2}(x)}\int\limits_{1/2}^\infty\int\limits_{S^{d-1}}\frac{\abs{v(y+d_x/4 r\theta )}}{r^{1+2s}}\mu(\d \theta)\d r\Big),
			\end{split}
		\end{equation}
		see, e.g., \cite[Theorem 3.3]{Gru24}. Without loss of generality, we assume that $x$ is closer to the boundary than 1. Let $x_0$ be the projection of $x$ onto the boundary. Since $w:=v(\cdot)-g(x_0)$ solves the same equation in the interior, we may replace $v$ by $w$ in \eqref{eq:interior}. Note that this does not change the left-hand side. Together with the generalized Hölder embedding, i.e., \autoref{lem:hoelder} applied with $\kappa(t):= 1+\int_t^1\omega(r)/r^{1+s}\d r$ and $\eps=1/4$, \eqref{eq:interior} yields
		\begin{equation}\label{eq:interior_refined}
			\begin{split}
				[u]_{C^{\sigma}(B_{d_x/4}(x))}&\le c_2\Big( \frac{d_x^s}{\kappa(d_x)}\norm{f}_{L^\infty(B_{d_x/2}(x))}+\frac{d_x^{-s}}{\kappa(d_x)}\norm{v-g(x_0)}_{L^\infty(B_{d_x/2}(x))}\\
				&\quad + (1-s)\frac{d_x^{-s}}{\kappa(d_x)}\esssup\limits_{y\in B_{d_x/2}(x)}\int\limits_{1/2}^\infty\int\limits_{S^{d-1}}\frac{\abs{v(y+d_x/4 r\theta )-g(x_0)}}{r^{1+2s}}\mu(\d \theta)\d r\Big),
			\end{split}
		\end{equation}
	
		We argue that the right-hand side of \eqref{eq:interior_refined} is bounded independently of $x\in \overline{\Omega}\cap B_{1/2}$. Firstly, the boundary estimates \eqref{eq:v_1-boundary_estimate} for $v_1$ and  \eqref{eq:v_2-boundary_estimate} for $v_2$ yield
		\begin{equation}\label{eq:help_boundary_estimate}
			\begin{split}
				\abs{v(y)-g(x_0)}&\le \abs{v_2(y)-g(x_0)}+\abs{v_1(y)-0}\\
				&\le c_4\Big(\norm{g}_{C^{\omega}_{\text{\normalfont ext} } (\Omega^c\cap \overline{B_2})}\kappa(d_x) + \norm{\tail_\mu^s[u]}_{L^\infty(\Omega\cap B_1)} \Big)d_x^s
			\end{split}
		\end{equation}
		Due to this estimate the term $d_x^{-s}\kappa(d_x)^{-1}\norm{v-g(x_0)}_{L^\infty(B_{d_x/2}(x))}$ is bounded.\smallskip
		
		Let us now consider the inhomogeneity $f$. By the previous estimate, we get
		\begin{equation*}
			\frac{d_x^s}{\kappa(d_x)}\norm{f}_{L^\infty(B_{d_x/2}(x))} \le \diam(\Omega)^s  \norm{\tail_\mu^s[u]}_{L^\infty(B_{1}\cap \Omega)}.
		\end{equation*}
				
		It remains to estimate the tail term in \eqref{eq:interior_refined}. For the moment, let $y\in B_{d_x/2}(x)$. Note that, if $y+(d_x/4)r\theta$ belongs to $\Omega$, then estimates very similar to \eqref{eq:help_boundary_estimate}, \eqref{eq:v_1-boundary_estimate}, and \eqref{eq:v_2-boundary_estimate} yield
		\begin{equation}\label{eq:boundary_helper1}
			\begin{split}
				&\abs{ v(y+(d_x/4)r\theta)-g(x_0) }\\
				&\quad\le c_5\Big(\norm{g}_{C^{\omega}_{\text{\normalfont ext} } (\Omega^c\cap \overline{B_2})}\kappa(\abs{y+(d_x/4)r\theta-x_0}) + \norm{\tail_\mu^s[u]}_{L^\infty(\Omega\cap B_1)} \Big)\abs{y+(d_x/4)r\theta-x_0}^s\\
				&\quad\le c_6\Big(\norm{g}_{C^{\omega}_{\text{\normalfont ext} } (\Omega^c\cap \overline{B_2})} +\norm{\tail_\mu^s[u]}_{L^\infty(\Omega\cap B_1)} \Big) \kappa(\abs{y-x_0}+(d_x/4)r)(\abs{y-x_0}+(d_x/4)r)^s\\
				&\quad\le c_7\Big(\norm{g}_{C^{\omega}_{\text{\normalfont ext} } (\Omega^c\cap \overline{B_2})} + \norm{\tail_\mu^s[u]}_{L^\infty(\Omega\cap B_1)}\Big)\kappa(d_x)d_x^s\big(1+r\big)^s.
			\end{split}			
		\end{equation}
		Here, we used \autoref{lem:almost_increasing} and the monotonicity of $\omega$. \smallskip
		
		If, instead, $y+(d_x/4)r\theta$ belongs to $\Omega^c\cap B_2$, then the assumption on the exterior datum yields
		\begin{equation}\label{eq:boundary_helper2}
			\begin{split}
			&\abs{ g(y+(d_x/4)r\theta)-g(x_0) }\\
			&\quad\le [g]_{C_{\text{ext}}^{\omega}}\omega\big(d_{y+d_x/4r\theta}+\abs{ y+d_x/4r\theta-x_0 }\big)\\
			&\quad\le c_8[g]_{C_{\text{ext}}^{\omega}}(3+r/2)^s d_x^s \kappa(d_x). 	
			\end{split}
		\end{equation}
		Here, we used \autoref{lem:sigma_vs_omega}.
		Lastly, if $y+(d_x/4)r\theta$ belongs to $B_2^c$, then 
		\begin{equation}\label{eq:boundary_helper3}
			d_x \frac{r}{4}\ge \abs{y+d_x \frac{r}{4}-\tilde{x}}- \abs{y-x}-\abs{x-\tilde{x}}\ge 3/4
		\end{equation}
		where $\tilde{x}$ is the center of the original ball $B_2$ fixed in the statement of \autoref{th:regularity}. These preliminary estimates allow us to bound the tail in \eqref{eq:interior_refined}. 
		\begin{align*}
		\frac{1-s}{d_x^{s}\kappa(d_x)}\esssup_{y\in B_{d_x/2}}\int\limits_{1/2}^\infty\int\limits_{S^{d-1}}\frac{\abs{v(y+d_x/4 r\theta )-g(x_0)}}{r^{1+2s}}\mu(\d \theta)\d r=\I{}+\II{}+\III{}
		\end{align*}
		where 
		\begin{align*}
			\I{}&:= \frac{1-s}{d_x^{s}\kappa(d_x)}\esssup_{y\in B_{d_x/2}}\int\limits_{1/2}^\infty\int\limits_{S^{d-1}}\1_{\Omega\cap B_2}(y+d_x/4 r\theta)\frac{\abs{v(y+d_x/4 r\theta )-g(x_0)}}{r^{1+2s}}\mu(\d \theta)\d r\\
			&\le \Lambda c_7\Big(\norm{g}_{C^{\omega}_{\text{\normalfont ext} } (\Omega^c\cap \overline{B_2})} + \norm{\tail_\mu^s[u]}_{L^\infty(\Omega\cap B_1)} \Big)\int\limits_{1/2}^\infty\frac{(1+r)^s}{r^{1+2s}}\d r.
		\end{align*}
		Here, we used \eqref{eq:boundary_helper1}. The second term $\II{}$ is defined via 
		\begin{align*}
			\II{}&:= \frac{1-s}{d_x^{s}\kappa(d_x)}\esssup_{y\in B_{d_x/2}}\int\limits_{1/2}^\infty\int\limits_{S^{d-1}}\1_{\Omega^c\cap B_2}(y+d_x/4 r\theta)\\
			&\qquad \times\frac{\abs{\eta(y+d_x/4r \theta)g(y+d_x/4 r\theta )-g(x_0)}}{r^{1+2s}}\mu(\d \theta)\d r\\
			&\le c_{9}[\eta g]_{C_{\text{ext}}^{\omega}(\Omega^c\cap \overline{B_2})} \Lambda\int\limits_{1/2}^\infty\frac{(3+r/2)^s}{r^{1+2s}}\d r \le \frac{c_{10}}{s}\norm{g}_{C_{\text{\normalfont ext}}^{\omega}(\Omega^c\cap \overline{B_2})}.
		\end{align*}
		Here, we used \eqref{eq:boundary_helper2} for $\eta g$ instead of $g$. The last term, i.e., $\III{}$, is bounded using \eqref{eq:boundary_helper3}. 
		\begin{align*}
			\III{}&:=\frac{1-s}{d_x^{s}\kappa(d_x)}\esssup_{y\in B_{d_x/2}}\int\limits_{1/2}^\infty\int\limits_{S^{d-1}}\1_{y+d_x/4 r\theta\in B_2^c}\frac{\abs{v(y+d_x/4 r\theta )-g(x_0)}}{r^{1+2s}}\mu(\d \theta)\d r\\
			&\le \frac{1-s}{d_x^{s}\kappa(d_x)}\esssup_{y\in B_{d_x/2}}\int\limits_{3/d_x}^\infty\int\limits_{S^{d-1}}\1_{y+d_x/4 r\theta\in B_2^c}\frac{\abs{v(y+d_x/4 r\theta )-g(x_0)}}{r^{1+2s}}\mu(\d \theta)\d r\\
			&=\frac{d_x^{s}}{4^{2s}}\frac{1-s}{\kappa(d_x)}\esssup_{y\in B_{d_x/2}}\int\limits_{3/4}^\infty\int\limits_{S^{d-1}}\frac{\abs{v(y+t\theta )-g(x_0)}}{t^{1+2s}}\mu(\d \theta)\d t\\
			&\le (1-s)\frac{\Lambda}{2s}\norm{g}_{L^\infty(\partial\Omega)}+\norm{\tail_\mu^s[v]}_{L^\infty(\Omega\cap B_1)}.
		\end{align*}
		This completes the estimate of $\III{}$ after acknowledging $\tail_\mu^s[v]\le \tail_\mu^s[u]$. \medskip
		
		So far, we have proven that a constant $C$ exists, depending only on $d, \lambda, \Lambda, \Omega, \omega(2)$, and a lower bound on $s$. This constant $C$ ensures that for any $x,y\in \overline{\Omega}\cap \overline{B_{1/2}}$ satisfying either $4\abs{x-y}\le \max\{ d_x,d_y \}$ or $x\in \Omega\cap \overline{B_{1/2}}$ and $y\in \partial\Omega\cap \overline{B_{1/2}}$ the inequality \eqref{eq:hoelder_estimate_general} holds. It is left to study the case when $x,y\in \Omega\cap \overline{B_{1/2}}$ such that $4\abs{x-y}\ge \max\{d_x,d_y\}$. Let $x_0$ be the projection of $x$ onto the boundary, and let $y_0$ be the projection of $y$. In this case, we simply estimate using 
		\begin{align*}
			\abs{u(x)-u(y)}&\le \abs{u(x)-g(x_0)} + \abs{g(x_0)-g(y_0)}+ \abs{u(y)-g(y_0)}\\
			&\le c_{11}\big(\sigma(d_x)+\sigma(d_y)\big) \Big(  \norm{g}_{C_{\text{\normalfont ext}}^{\omega}(B_2\cap \Omega^c)} +\norm{\tail_\mu^s[u]}_{L^\infty(\Omega\cap B_1)}\Big)\\
			&\quad + \omega(\abs{x_0-y_0})[g]_{C^\omega(\partial\Omega)}.
		\end{align*}
		Note that, due to \autoref{lem:almost_increasing} and \autoref{lem:sigma_vs_omega}, 
		\begin{align*}
			\omega(\abs{x_0-y_0})&\le \omega(\abs{x_0-x}+\abs{x-y}+\abs{y-y_0})\le  \omega(9\abs{x-y})\\
			&\le c_{12} \sigma(9\abs{x-y})\le 9c_{12} \sigma(\abs{x-y})
		\end{align*}
		and 
		\begin{equation*}
			\sigma(\abs{x-x_0})+\sigma(\abs{y-y_0})\le 2\sigma(4\abs{x-y})\le 8\sigma(\abs{x-y}).
		\end{equation*}
		Thus, we have proven \eqref{eq:hoelder_estimate_general} for all combinations of $x,y\in \overline{\Omega\cap B_{1/2}}$. 		
	\end{proof}
	
	\section{Counterexamples}\label{sec:cex}
	The goal in this section is to provide examples that demonstrate the sharpness of the assertion in \autoref{th:regularity} and \autoref{cor:regularity}. We provide the proof of \autoref{cex:counterexample_cs_regularity} and \autoref{th:counterexample_general}.
	
	\begin{proof}[{Proof of \autoref{th:counterexample_general}}]
		We pick $z:= e_1$, $g(y):= \omega(\abs{y'})\eta(\abs{y})$ where $y=(y_1, y')$ and $\eta$ is a smooth and nonnegative cutoff function such that $\eta=1$ on $[1,4]$ and $\eta=0$ outside of a slightly larger interval. We find an explicit expression for the solution $u$ to \eqref{eq:inhom_dirichlet_ball}, e.g., see \cite{Buc16}. We consider $x=te_1\in B_1(0)$, $0<t<1$. Up to a constant multiple, $c_{d,s}=\Gamma(d/2)
		\sin(\pi s)/(\pi^{d/2+1})$, the difference $\abs{u(te_1)-g(e_1)}$ at $te_1$ equals
		\begin{align*}
			\int\limits_{B_1(0)^c} \frac{(1-\abs{te_1}^2)^s}{(\abs{y}^2-1)^s} \frac{g(y)}{\abs{te_1-y}^d}\d y\ge \frac{1}{5^s} \int\limits_{B_4\setminus B_1(0)} \frac{(1-t)^s}{(\abs{y_1}-\sqrt{1-\abs{y'}^2})^s} \frac{\omega(\abs{y'})}{\big((t-y_1)^2+\abs{y'}^2\big)^{d/2}}\d (y_1,y').
		\end{align*}  
		Here, we used $\sqrt{1-\abs{y'}^2}+\abs{y_1}\le 5$. We start by calculating the integral with respect to $y_1$. To ease the notation, we write $a:= \sqrt{1-\abs{y'}^2}$. We start by proving the following claim.\medskip
		
		\textbf{Claim A.} If $1-t\le \abs{y'}\le 1$, then 
		\begin{equation*}
			(1-s)\int_{a}^2 \frac{1}{(y_1-a)^{s}}\frac{1}{\big((y_1-t)^2+1-a^2\big)^{d/2}}\d y_1\ge 8^{-d} \abs{y'}^{1-d-s}. 
		\end{equation*}
		
		Firstly, we use the change of variables $r=y_1-a$ after which it suffices to find a lower bound on the integral 
		\begin{equation*}
			(1-s)\int_{0}^1 \frac{1}{r^{s}}\frac{1}{\big((r+a-t)^2+1-a^2\big)^{d/2}}\d r.
		\end{equation*}		
		Now, we make the following observations
		\begin{align*}
			(r+a-t)^2+1-a^2\le 4\big( r^2+(1-a)^2+(1-t)^2+1-a^2 \big)
		\end{align*}
		and $(1-a)^2\le (1-a^2)^2\le 1-a^2$ which yield 
		\begin{equation*}
			(r+a-t)^2+1-a^2\le 4^2(r+\abs{y'})^2.
		\end{equation*}
		In the last inequality, we used the definition of $a$ and the estimate $1-t\le \abs{y'}$. This observation yields
		\begin{align*}
				(1-s)\int_{a}^2 \frac{1}{(y_1-a)^{s}}\frac{1}{\big((y_1-t)^2+1-a^2\big)^{d/2}}\d y_1&\ge 4^{-d} (1-s)\int_{0}^1 \frac{1}{r^{s}}\frac{1}{(r+\abs{y'})^{d}}\d r\\
				= 4^{-d}\frac{1}{(1+\abs{y'})^{d}}+d4^{-d} \int_{0}^1 \frac{r^{1-s}}{(r+\abs{y'})^{d+1}}\d r
				&\ge 4^{-d}\frac{1}{(1+\abs{y'})^{d}}+d4^{-d} \int_{\abs{y'}}^1 \frac{\abs{y'}^{1-s}}{(r+\abs{y'})^{d+1}}\d r\\
				&= 4^{-d}\frac{1-\abs{y'}}{(1+\abs{y'})^{d}}+8^{-d}\abs{y'}^{1-d-s}.
		\end{align*}
		This yields the claim. \medskip
		
		We resume our original calculation and apply claim A:
		\begin{align*}
			\int\limits_{B_1(0)^c} \frac{(1-\abs{te_1}^2)^s}{(\abs{y}^2-1)^s} \frac{g(y)}{\abs{te_1-y}^d}\d y\ge \frac{8^{-d}}{1-s}\int\limits_{B_{1}^{d-1}\setminus B_{1-t}^{d-1}(0)} \frac{\omega(\abs{y'})}{\abs{y'}^{d+s-1}}\d y'= \frac{\omega_{d-2}8^{-d}}{1-s}\int_{1-t}^1 \frac{\omega(r)}{r^{1+s}}\d r.
		\end{align*}
		This yields the desired lower bound on $\abs{u(te_1)-g(e_1)}$, since $c_{d,s}/(8^d(1-s))$ is bounded from below by a positive constant that depends only on $d$ and $s_0$. 
	\end{proof}
	
	\begin{proof}[{Proof of \autoref{cex:counterexample_cs_regularity}}]
		Note that \autoref{cex:counterexample_cs_regularity} is a direct consequence of \autoref{th:counterexample_general} by choosing $\omega(t)=t^s\iota(t)$ and due to the failure of the Dini property of $\iota$.
	\end{proof}
	
	In the case of one spatial dimension, the solutions to the inhomogeneous Dirichlet problem for the Laplacian behave differently. Clearly, the solution $u:[a,b]\to \R$ to $u''=0$ in $(a,b)$ subject to the boundary conditions $u(a)=A$ and $u(b)=B$, $a<b$, $A,B\in \R$ is smooth up to the boundary. This is in contrast to the fractional problem as seen in the following proposition.
	
	\begin{proposition}\label{prop:counterexample_general_d1}
		For any arbitrary nondecreasing function $\omega:[0,\infty)\to [0,\infty)$ with $\omega(0)=0$, we find an exterior datum $g\in C_c^{\omega}((-1,1)^c)$ such that the solution $u$ to \eqref{eq:inhom_dirichlet_ball} satisfies
		\begin{equation*}
			\abs{u(x)-g(1)}\ge \frac{\pi}{8} s(1-s)(1-x)^s \int\limits_{1-x}^1 \frac{\omega(r)}{r^{1+s}}\d r
		\end{equation*}
		for $0\le x\le 1$.
	\end{proposition}
	
	This result in \autoref{prop:counterexample_general_d1} is not robust as $s\to 1-$, and it cannot be, as noted previously.
	
	\begin{proof}[{Proof of \autoref{prop:counterexample_general_d1}}]
		We set $g(x):= \omega(x-1)\1_{[1,3]}(x)$ for $x\in \R$. By the Poisson representation, we find
		\begin{align*}
			u(x)-g(1)&= c_{1,s}\int_{1}^3\frac{(1-x^2)^s}{(y^2-1)^s}\frac{\omega(y-1)}{y-x}\d y\ge \frac{c_{1,s}}{4^s}(1-x)^s\int_{1-x}^1\frac{\omega(y)}{y^s(y+1-x)}\d y\\
			&\ge \frac{c_{1,s}}{8}(1-x)^s\int_{1-x}^1 \frac{\omega(y)}{y^{1+s}}\d y
		\end{align*} 
		by noting that for $y\in [1-x,1]$, one has $y+1-x\le 2y$. Finally, note that 
		\begin{equation*}
			c_{1,s}= \Gamma(1/2)\sin(\pi s)/(\pi^{3/2})\ge \frac{s(1-s)}{\pi}.
		\end{equation*}
	\end{proof}
	
	Now, we provide the example announced after \autoref{rem:uniformity}
	
	\begin{example}\label{ex:countercounterexample}
		Let $d\ge 2$. We define $g$ very similarly to what is constructed in the proof of \autoref{cex:counterexample_cs_regularity}. We set 
		\begin{equation*}
			g(y):= y_2 \abs{y_2}^{s-1} \eta(\abs{y}),
		\end{equation*} 
		where, again, $\eta$ is a cutoff function and $y=(y_1,y_2, \dots, y_d)$. With obvious modifications in the case $d=2$. In contrast to the choice of $g$ in \autoref{cex:counterexample_cs_regularity}, this function changes signs. Due to this sign change, the solution $u$ benefits from some cancellation. Note, in particular, that $u(te_1)= 0$. This is in direct contrast to the solution constructed in \autoref{th:counterexample_general}. A more detailed analysis shows that $u\in C^{s}(\overline{B_1})$.
	\end{example}
	

\end{document}